\documentclass[11pt,twoside]{article}
\usepackage{mathrsfs}
\usepackage{amsmath}
\usepackage{amsthm}
\input amssymb.sty
\usepackage{amsfonts}
\usepackage{amssymb}
\usepackage{latexsym}
\usepackage[all]{xy}

\date{\empty}
\allowdisplaybreaks[4] \footskip=15pt

\numberwithin{equation}{section} \theoremstyle{plain}
\newtheorem*{thm*}{Main Theorem}
\newtheorem{theorem}{Theorem}
\newtheorem{corollary}[theorem]{Corollary}
\newtheorem*{corollary*}{Corollary}

\newtheorem*{claim*}{Claim}
\newtheorem{lemma}[theorem]{Lemma}
\newtheorem*{lemma*}{Lemma}

\newtheorem*{proposition*}{Proposition}
\newtheorem{remark}[theorem]{Remark}
\newtheorem*{remark*}{Remark}

\newtheorem*{example*}{Example}

\newtheorem*{question*}{Question}

\newtheorem*{definition*}{Definition}

\newtheorem*{acknowledgements*}{ACKNOWLEDGEMENTS}

\begin{document}
\begin{center}
{\large \bf Projections and Idempotents in $*$-reducing Rings Involving the Moore-Penrose Inverse}\\

\vspace{0.8cm} {\small \bf Xiaoxiang Zhang, \ \ Shuangshuang Zhang, \ \ Jianlong Chen\footnote{Corresponding author.\\
E-mail: z990303@seu.edu.cn (X. Zhang), jlchen@seu.edu.cn (J. Chen).},\ \ Long Wang}\\
\vspace{0.6cm} {\rm Department of Mathematics, Southeast University, Nanjing, 210096, China}
\end{center}

\bigskip

{ \bf  Abstract:}
\leftskip0truemm\rightskip0truemm
In [
Comput. Math. Appl. 59 (2010) 764-778],
Baksalary and Trenkler characterized some complex idempotent matrices of the form
$(I-PQ)^{\dag}P(I-Q)$, $(I-QP)^{\dag}(I-Q)$ and $P(P+Q-QP)^{\dag}$
in terms of the column spaces and null spaces of $P$ and $Q$,
where $P, Q\in \mathbb{C}^{\textsf{OP}}_n = \{L\in \mathbb{C}_{n,n}\ |\ L^2 = L = L^*\}$.
We generalize these results from $\mathbb{C}_{n,n}$ to any $*$-reducing rings.
 \bigskip


\section { \bf Introduction}

Before 2010, several results scattered in the literature express an idempotent having given onto and along spaces in terms of a pair of projections
in various settings (see, e.g., \cite{Afriat, Giol, Greville, Kovarik, Vidav}).
Their common assumption is the invertibility of certain functions of the involved projections.
Recently, these results were unified and reestablished by Baksalary and Trenkler \cite{Baksalary Trenkler}
in a generalized form in a complex Euclidean vector space,
where Moore-Penrose inverse was involved instead of the ordinary inverse.
Among others, some complex idempotent matrices of the form
$(I-PQ)^{\dag}P(I-Q)$, $(I-QP)^{\dag}(I-Q)$ and $P(P+Q-QP)^{\dag}$
are characterized in terms of the column spaces and null spaces of $P$ and $Q$,
where $P, Q\in \mathbb{C}^{\textsf{OP}}_n = \{L\in \mathbb{C}_{n,n}\ |\ L^2 = L = L^*\}$.

The purpose of this paper is to generalize the aforementioned results in \cite{Baksalary Trenkler}
from $\mathbb{C}_{n,n}$ to any $*$-reducing rings.
This is achieved based on the characterizations of the Moore-Penrose inverse of the differences and the products of projections
obtained in \cite{Zhang Zhang Chen Wang}.
The main idea behind our argument comes from \cite{Baksalary Trenkler}.
However, the methods based on decomposition or rank of matrix are not available in an arbitrary $*$-reducing ring.
The results in this paper are proved by a purely ring theoretical method.

Throughout this paper, $R$ is an associative ring with unity and an involution $a \mapsto a^*$ satisfying
$(a^*)^* = a$, $(a+b)^* = a^* + b^*$, $(ab)^* = b^*a^*$.
An element $a\in R$ has Moore-Penrose inverse (or MP inverse for short),
if there exists $b$ such that the following equations hold \cite{Penrose}:
$$(1)\ \ aba = a, \quad\quad  (2)\ \ bab = b, \quad\quad  (3)\ \ (ab)^* = ab, \quad\quad (4)\ \ (ba)^{\ast} = ba.$$
In this case, $b$ is unique and denoted by $a^{\dag}$.
Moreover, we have $(a^{\dag})^{\dag} = a$.
It is easy to see that $a$ has MP inverse if and only if $a^{*}$ has MP inverse, and in this case $(a^{*})^{\dag} = (a^{\dag})^{*}$.

We write $R^{-1}$ and $R^{\dag}$ as the set of all invertible elements and all MP invertible elements in $R$, respectively.
If $a\in R^{\dag}$, then $a^{*}a$, $aa^{*}\in R^{\dag}$ and the following equalities hold:
$(a^{*}a)^{\dag} = a^{\dag}(a^{*})^{\dag}$, $(aa^{*})^{\dag} = (a^{*})^{\dag}a^{\dag}$ and
$a^{\dag} = (a^{*}a)^{\dag}a^{*} = a^{*}(aa^{*})^{\dag}.$
(See, e.g., \cite[Lemma 2.1]{Benitez CvetkovicIlic}.)

If $a^{*} = a\in R^{\dag}$, then $aa^{\dag} = a^{\dag}a$.
An idempotent $p \in R$ is called a projection if it is self-adjoint, i.e., $p^{*} = p$.

Recall from \cite{Koliha Patricio} that a ring $R$ is said to be $*$-reducing if, for any element $a\in R$,
$a^*a = 0$ implies $a = 0$.
Note that $R$ is $*$-reducing if and only if the following implications hold for any $a\in R$:
$$a^{*}ax = a^{*}ay \Rightarrow ax = ay \quad\quad \mbox{and} \quad\quad xaa^{*} = yaa^{*} \Rightarrow xa = ya.$$
It is well-known that any $C^{*}$-algebra is a $*$-reducing ring.

Let $X$ and $Y$ be two nonempty subset of $R$.
We write $X\bot Y$ to indicate that $y^{*}x = 0$ for all $(x, y)\in X\times Y$.
Suppose that $X$ and $Y$ are right ideals of a $*$-reducing ring $R$.
Then $X\bot Y$ implies $X\cap Y = \{0\}$.
In particular, if $X + Y = R$ and $X\bot Y$, then we write $X \oplus^{\bot} Y = R$.

\section{ \bf  Main results}

Let us remind the reader that, in what follows, $R$ is always a ring with involution $*$.
By $p$ and $q$ we mean two projections in $R$.
We also fix the notations $a = pqp$, $b = pq(1-p)$, $d = (1-p)q(1-p)$, $\overline{p} = 1-p$ and $\overline{q} = 1-q$.

The following lemmas will be used in the sequel.

\begin{lemma}\label{Lemma 2.2}
\emph{(1)} $bb^{*} = (p-a)-(p-a)^{2}$.

\emph{(2)}  $b^{*}b =d-d^{2}$.
\end{lemma}
\begin{proof}
By a direct verification.
\end{proof}

\begin{lemma}\label{Lemma 2.3}
\emph{(1)} If $p\overline{q}\in R^{\dag}$, then $(p-a)(p-a)^{\dag}b = b$.

\emph{(2)} If $\overline{p}q\in R^{\dag}$, then $bdd^{\dag} = b$.

\emph{(3)} If $p\overline{q}, \overline{p}q\in R^{\dag}$, then $bd^{\dag} = (p-a)^{\dag}b$ and $d^{\dag}b^{*} = b^{*}(p-a)^{\dag}$.

\emph{(4)} If $p\overline{q}, \overline{p}q\in R^{\dag}$, then $p-q \in R^{\dag}$.
\end{lemma}
\begin{proof}
See \cite[Lemma 2.3 and Theorem 4.1(iii)]{Benitez CvetkovicIlic} or \cite[Lemma 3]{Zhang Zhang Chen Wang}.
\end{proof}

\begin{lemma}\label{corollary 2.5}
The following conditions are equivalent for any two projections $p$ and $q$ in a $*$-reducing ring $R$:

\emph{(1)} $1-pq\in R^{\dag}$, \quad \emph{(2)} $p-pqp\in R^{\dag}$, \quad \emph{(3)} $p-pq \in R^{\dag}$, \quad  \emph{(4)} $p-qp\in R^{\dag}$,

\emph{(5)} $1-qp\in R^{\dag}$, \quad \emph{(6)} $q-pq\in R^{\dag}$,  \hspace{4.7mm} \emph{(7)} $\overline{p}q\overline{p}\in R^{\dag}$, \hspace{7.2mm}
\emph{(8)} $p+q-qp\in R^{\dag}$. \\
Moreover, when any one of these conditions is satisfied we have the following equalities:

\emph{(i)} $(p-pqp)^{\dag} = (1-pq)^{\dag}p$;

\emph{(ii)} $(1-pq)^{\dag} = (p-a)^{\dag}(1+b) + 1-p$;

\emph{(iii)} $(1-pq)(1-pq)^{\dag} = (1-pq)^{\dag}(1-pq) = (p-a)(p-a)^{\dag}+1-p$.
\end{lemma}
\begin{proof}
See \cite[(2.9), (2.9), (2.13), Corollary 5 and 7]{Zhang Zhang Chen Wang}.
\end{proof}

The following lemma has been included in the proof of \cite[Lemma 2.3(i)]{Benitez CvetkovicIlic}.

\begin{lemma}\label{Lemma 3.1}
If $p$ and $q$ are two projections in $R$ such that $pq\in R^{\dag}$,
then $pqp(pqp)^{\dag}pq = pq$.
\end{lemma}
\begin{proof}
Note that $(pqp)^{\dag} = [(pq)^{\dag}]^*(pq)^{\dag}$ since $pq\in R^{\dag}$.
It is easy to verify that $pqp(pqp)^{\dag}pq = pq$.
\end{proof}

The next lemma is of interest in its own right.

\begin{lemma}\label{Lemma 3.2}
Let $p$ and $q$ be two projections in $R$.

\emph{(1)} If $\overline{p}q \in R^{\dag}$,
           then $(\overline{p}q)^{\dag} = q(\overline{p}q\overline{p})^{\dag}$.

\emph{(2)} If $\overline{p}q\in R^{\dag}$, then $x = p+\overline{p}(\overline{p}q)^{\dag}$ is a projection and $xR = pR+qR$.

\emph{(3)} If $p\overline{q}\in R^{\dag}$, then $y = p-p(p\overline{q})^{\dag}$ is a projection and $yR = pR\cap qR$.

\emph{(4)} Suppose that $\overline{p}q\in R^{\dag}$, then $pR+qR = R$ if and only if $\overline{p}q\overline{p}R = \overline{p}R$.
\end{lemma}
\begin{proof}
(1) See \cite[Theorem 4.1(ii)]{Benitez CvetkovicIlic}.

(2) This was proved in \cite[Lemma 4.2(i)]{Benitez CvetkovicIlic}, where $(p+q)^{\dag}$ was involved.
    We provide here another argument.

    Since $x = p+\overline{p}(\overline{p}q)^{\dag} = p+\overline{p}q(\overline{p}q\overline{p})^{\dag}
             = p+\overline{p}q\overline{p}(\overline{p}q\overline{p})^{\dag}$,
    we have $x^2 = x = x^*$.

    Moreover, by Lemma \ref{Lemma 3.1} it follows that
    \begin{eqnarray*}
      x & = & p+\overline{p}q\overline{p}(\overline{p}q\overline{p})^{\dag}
          =   p+\overline{p}q\overline{p}(\overline{p}q\overline{p})^{\dag}-q(1-p)+q(1-p)\\
        & = & p+\overline{p}q\overline{p}(\overline{p}q\overline{p})^{\dag}
              -q(1-p)q(1-p)(\overline{p}q\overline{p})^{\dag}+q(1-p)\quad\quad (\mbox{see Lemma \ref{Lemma 3.1}})\\
        & = & p+(1-q)\overline{p}q\overline{p}(\overline{p}q\overline{p})^{\dag}+q(1-p)\\
        & = &  p-(1-q)pq\overline{p}(\overline{p}q\overline{p})^{\dag}+q(1-p)\\
        & = &  p-pq\overline{p}(\overline{p}q\overline{p})^{\dag}+qpq\overline{p}(\overline{p}q\overline{p})^{\dag}+q(1-p)\\
        & = & p[1-q\overline{p}(\overline{p}q\overline{p})^{\dag}]+q[pq\overline{p}(\overline{p}q\overline{p})^{\dag}+(1-p)].
    \end{eqnarray*}
    Hence $xR \subseteq pR + qR$.
    Meanwhile, for any $r_{1}, r_{2}\in R$, we have
    \begin{eqnarray*}
       pr_{1}+qr_{2} & = & pr_{1}+(pq+\overline{p}q)r_{2} = p(r_{1}+qr_{2})+\overline{p}qr_{2}\\
                     & = & p(pr_{1}+qr_{2})+(\overline{p}q\overline{p})
                           (\overline{p}q\overline{p})^{\dag}\overline{p}qr_{2}\quad\quad (\mbox{see Lemma \ref{Lemma 3.1}})\\
                     & = & p(pr_{1}+qr_{2})+(\overline{p}q\overline{p})(\overline{p}q\overline{p})^{\dag}\overline{p}(pr_{1}+qr_{2})\\
                     & = & p(pr_{1}+qr_{2})+(\overline{p}q\overline{p})(\overline{p}q\overline{p})^{\dag}(pr_{1}+qr_{2})\\
                     & = & [p+(\overline{p}q\overline{p})(\overline{p}q\overline{p})^{\dag}](pr_{1}+qr_{2}) = x(pr_{1}+qr_{2}),
    \end{eqnarray*}
    from which one can see that $pR + qR \subseteq xR$.

(3) See \cite[Lemma 4.2(ii)]{Benitez CvetkovicIlic}.

(4) ``$\Rightarrow$" First, $R = pR+qR = [p+\overline{p}(\overline{p}q)^{\dag}]R$ by (2).
                     Hence
                     \begin{eqnarray}
                           1 = [p+\overline{p}(\overline{p}q)^{\dag}]r = pr+\overline{p}q(\overline{p}q\overline{p})^{\dag}r
                     \end{eqnarray}
                     for some $r\in R$.
                     Left-multiplying (2.1) by $\overline{p}q\overline{p}(\overline{p}q\overline{p})^{\dag}$ we obtain
                     $\overline{p}q\overline{p}(\overline{p}q\overline{p})^{\dag} = \overline{p}q\overline{p}(\overline{p}q\overline{p})^{\dag}r$.
                     Similarly, right-multiplying (2.1) by $p$ one can get $p = pr$.
                     Whence
                     \begin{eqnarray*}
                           (\overline{p}q\overline{p})(\overline{p}q\overline{p})^{\dag}
                           = (\overline{p}q\overline{p})^{\dag}(\overline{p}q\overline{p})r
                           \overset{(2.1)}{=\!\!=\!\!=\!\!=} 1-pr
                           = 1-p.
                     \end{eqnarray*}
                     Consequently, $\overline{p}q\overline{p}R = \overline{p}q\overline{p}(\overline{p}q\overline{p})^{\dag}R
                                                               = \overline{p}R$.

    ``$\Leftarrow$" Since $\overline{p}q\in R^{\dag}$, we have $\overline{p}q\overline{p} \in R^{\dag}$.
                    If $\overline{p}R = \overline{p}q\overline{p}R
                                      = \overline{p}q\overline{p}(\overline{p}q\overline{p})^{\dag}R$,
                    then $1-p = \overline{p}q\overline{p}(\overline{p}q\overline{p})^{\dag}$
                    by \cite[Lemma 4.1(4)]{Benitez CvetkovicIlic}.
                    According to (1), we have
                    $1 = p+1-p = p+\overline{p}q\overline{p}(\overline{p}q\overline{p})^{\dag} = p+\overline{p}(\overline{p}q)^{\dag}$.
                    Therefore, $R = [p+\overline{p}(\overline{p}q)^{\dag}]R = pR+qR$ follows by (2).
\end{proof}

\begin{remark}\label{Remark 1}
\emph{Let $I$ be a right ideal of $R$.
If $I$ is generated by a projection,
then there exists a unique projection $p\in R$ such that $I = pR$ (see, e.g., \cite[Lemma 4.1(4)]{Benitez CvetkovicIlic}).
Therefore, in Lemma \ref{Lemma 3.2}, the projection $x$ such that $xR = pR+qR$
(resp., $y$ such that $yR = pR\cap qR$) is unique.}
\end{remark}

\begin{lemma}\label{Lemma 3.3}
If $e^2 = e\in R^{\dag}$, then $ee^{\dag}R = eR$ and $(1-e^{\dag}e)R = (1-e)R$.
\end{lemma}
\begin{proof}
It is easy to see that $ee^{\dag}R = eR$ since $ee^{\dag}e = e$.
On the other hand, $1-e^{\dag}e = (1-e)(1-e^{\dag}e)$ and $1-e = (1-e^{\dag}e)(1-e)$ imply $(1-e^{\dag}e)R = (1-e)R$.
\end{proof}

The following theorem and its corollary generalize their counterpart in \cite{Baksalary Trenkler}.

\begin{theorem}\label{Theorem 3.4}
Let $p$ and $q$ be projections in a $*$-reducing ring $R$ such that $1-pq\in R^{\dag}$
Then

\emph{(1)} $(\overline{q}p)^{\dag} = (1-pq)^{\dag}p\overline{q}$ is an idempotent;

\emph{(2)} $(\overline{q}p)^{\dag}R = pR\cap (\overline{p}R+\overline{q}R)$;

\emph{(3)} $[1-(\overline{q}p)^{\dag}]R = (\overline{p}R\cap\overline{q}R)\oplus^{\bot} qR$.
\end{theorem}

\begin{proof}
(1) Since $1-pq\in R^{\dag}$, we have
    $\overline{q}p\in R^{\dag}$, $p\overline{q}p\in R^{\dag}$ and $(p\overline{q}p)^{\dag} = (1-pq)^{\dag}p$
    by Lemma \ref{corollary 2.5}(1)$\Leftrightarrow$(2)$\Leftrightarrow$(4).
    In view of Lemma \ref{Lemma 3.2}(1), one can see that
    $(\overline{q}p)^{\dag} = (p\overline{q}p)^{\dag}\overline{q} = (1-pq)^{\dag}p\overline{q}$.
    Moreover, it follows that
    $(\overline{q}p)^{\dag}(\overline{q}p)^{\dag} = (p\overline{q}p)^{\dag}\overline{q}(p\overline{q}p)^{\dag}\overline{q}
              = (p\overline{q}p)^{\dag}\overline{q} = (\overline{q}p)^{\dag}$.

(2) First, $p\overline{q}\in R^{\dag}$ follows by the hypothesis and Lemma \ref{corollary 2.5}(1)$\Leftrightarrow$(3).
    Now replacing $p$ and $q$ by $\overline{p}$ and $\overline{q}$ respectively in Lemma \ref{Lemma 3.2}(2),
    one can see that
    \begin{eqnarray}
           q^{\prime} = \overline{p}+p(p\overline{q})^{\dag}
    \end{eqnarray}
    is a projection and
    \begin{eqnarray}
           \overline{p}R+\overline{q}R = q^{\prime}R.
    \end{eqnarray}
    Whence $p\overline{q^{\prime}} = p[p-p(p\overline{q})^{\dag}] = p-p(p\overline{q})^{\dag} = \overline{q^{\prime}}$
    and $(p\overline{q^{\prime}})^{\dag} = \overline{q^{\prime}}$.
    Replacing $q$ by $q^{\prime}$ in Lemma \ref{Lemma 3.2}(3),
    we can see that $p-p(p\overline{q^{\prime}})^{\dag}$ is projection and
    \begin{eqnarray}
          (p-p(p\overline{q^{\prime}})^{\dag})R = pR\cap q^{\prime}R = pR\cap (\overline{p}R+\overline{q}R).
    \end{eqnarray}
    Let $y^{\prime} = p-p(p\overline{q^{\prime}})^{\dag}$, then
    $y^{\prime} = p-\overline{q^{\prime}} = p(p\overline{q})^{\dag}
                = p\overline{q}(p\overline{q}p)^{\dag} = (p\overline{q}p)^{\dag}p\overline{q}p.$
    According to Lemma \ref{Lemma 3.2}(1),
    $(\overline{q}p)^{\dag} = [(p\overline{q})^{\dag}]^{*} = (p\overline{q}p)^{\dag}\overline{q}$.
    This implies
    $(\overline{q}p)^{\dag}(\overline{q}p) = (p\overline{q}p)^{\dag}\overline{q}p = (p\overline{q}p)^{\dag}p\overline{q}p$
    and hence
    $(\overline{q}p)^{\dag}R = (\overline{q}p)^{\dag}(\overline{q}p)R = (p\overline{q}p)^{\dag}p\overline{q}pR
                                 = y^{\prime}R = pR\cap (\overline{p}R+\overline{q}R).$

(3) Note that $1-pq\in R^{\dag}$ is equivalent to $\overline{p}q\in
R^{\dag}$ (see Lemma \ref{corollary 2.5} (1)$\Leftrightarrow$(6)).
    Replacing $p$ and $q$ by $\overline{p}$ and $\overline{q}$ respectively in Lemma \ref{Lemma 3.2}(3),
    we obtain that
    \begin{eqnarray}
    p^{\prime} = \overline{p}-\overline{p}(\overline{p}q)^{\dag}
    \end{eqnarray}
    is a projection and
    \begin{eqnarray}
          \overline{p}R\cap\overline{q}R = p^{\prime}R.
    \end{eqnarray}
    By Lemma \ref{Lemma 3.1} and Lemma \ref{Lemma 3.2}(1),
    \begin{eqnarray*}
          p^{\prime}q & = & [\overline{p}-\overline{p}(\overline{p}q)^{\dag}]q
                        =   \overline{p}q-\overline{p}q(\overline{p}q\overline{p})^{\dag}q \ \quad\quad(\mbox{see Lemma \ref{Lemma 3.2}(1)})\\
                      & = & \overline{p}q-\overline{p}q = 0,  \quad\quad\quad\quad\quad\quad\quad\quad\quad\quad(\mbox{see Lemma \ref{Lemma 3.1}})
    \end{eqnarray*}
    hence
    \begin{eqnarray}
          p^{\prime}R\bot qR.
    \end{eqnarray}
    Moreover
    $\overline{p^{\prime}}q = [p+\overline{p}(\overline{p}q)^{\dag}]q = pq+\overline{p}(\overline{p}q)^{\dag}q
                            = pq+\overline{p}q(\overline{p}q\overline{p})^{\dag}q = pq+\overline{p}q = q$.
    This implies $(\overline{p^{\prime}}q)^{\dag} = q$.
    Now, let $x^{\prime} = p^{\prime}+\overline{p^{\prime}}(\overline{p^{\prime}}q)^{\dag}$.
    Replacing $p$ by $p^{\prime}$ in Lemma \ref{Lemma 3.2}(2),
    we can see that
    $x^{\prime}R = p^{\prime}R+qR \overset{(2.6)}{=\!\!\!=\!\!\!=\!\!\!=\!\!\!=}(\overline{p}R\cap\overline{q}R)+ qR$.
    Furthermore, it follows from (2.7) that
    $x^{\prime}R = (\overline{p}R\cap\overline{q}R)+ qR = (\overline{p}R\cap\overline{q}R)\oplus^{\bot} qR$.

    On the other hand, we have
    \begin{eqnarray*}
           1-(\overline{q}p)(\overline{q}p)^{\dag} & = & 1-(\overline{q}p)(p\overline{q}p)^{\dag}\overline{q}
                                                         \quad\quad\quad\quad\quad\quad\quad(\mbox{see Lemma \ref{Lemma 3.2}(1)})            \\
                                                   & = & 1-[p\overline{q}p(p\overline{q}p)^{\dag}\overline{q}+
                                                         (1-p)\overline{q}p(p\overline{q}p)^{\dag}\overline{q}]                              \\
                                                   & = & 1-[p\overline{q}-(1-p)qp(p\overline{q}p)^{\dag}\overline{q}]
                                                         \quad\quad(\mbox{see Lemma \ref{Lemma 3.1}})                                        \\
                                                   & = & 1-[p\overline{q}-(\overline{p}q\overline{p})^{\dag}\overline{p}qp\overline{q}]
                                                         \quad\quad\quad\quad\ \ (\mbox{see Lemma \ref{Lemma 2.3}(3)})                       \\
                                                   & = & 1-[p\overline{q}+(\overline{p}q\overline{p})^{\dag}\overline{p}q\overline{p}(1-q)]  \\
                                                   & = & 1-[p\overline{q}+(\overline{p}q\overline{p})^{\dag}\overline{p}q\overline{p}-
                                                         (\overline{p}q\overline{p})^{\dag}\overline{p}q\overline{p}q]                       \\
                                                   & = & 1-[p\overline{q}+(\overline{p}q\overline{p})^{\dag}\overline{p}q\overline{p}
                                                         -\overline{p}q]\quad\quad\quad(\mbox{see Lemma \ref{Lemma 3.1}})                    \\
                                                   & = & 1-p+q-(\overline{p}q\overline{p})^{\dag}\overline{p}q\overline{p} = \overline{p}-\overline{p}(\overline{p}q)^{\dag}+q \\
                                                   & = & p^{\prime}+q = p^{\prime}+\overline{p^{\prime}}(\overline{p^{\prime}}q)^{\dag} = x^{\prime},
   \end{eqnarray*}
   from which one can get
   $[1-(\overline{q}p)(\overline{q}p)^{\dag}]R = x^{\prime}R = (\overline{p}R\cap\overline{q}R)\oplus^{\bot} qR$.
   Finally, in view of Lemma \ref{Lemma 3.3}, it follows that
   $[1-(\overline{q}p)^{\dag}]R = [1-(\overline{q}p)(\overline{q}p)^{\dag}]R = (\overline{p}R\cap\overline{q}R)\oplus^{\bot} qR$.
\end{proof}

\begin{corollary}\label{Corollary 3.5}
Let $p$ and $q$ be projections in a $*$-reducing ring $R$ such that $1-pq\in R^{\dag}$ and $pR\cap qR = {0}$, then

\emph{(1)} $1-pq\in R^{-1}$ and $(\overline{q}p)^{\dag} = (1-pq)^{-1}p\overline{q}$;

\emph{(2)} $(\overline{q}p)^{\dag}R = pR$;

\emph{(3)} $[1-(\overline{q}p)^{\dag}]R = (\overline{p}R\cap\overline{q}R)\oplus^{\bot} qR$.
\end{corollary}

\begin{proof}
(1) By Lemma \ref{corollary 2.5} (1)$\Leftrightarrow$(3), $1-pq\in R^{\dag}$ implies $p\overline{q}\in R^{\dag}$.
    Then $[p-p(p\overline{q})^{\dag}]R = pR\cap qR = 0$ by Lemma \ref{Lemma 3.2}(3).
    According to Lemma \ref{Lemma 3.2}(1), we have $p\overline{q}(p\overline{q}p)^{\dag} = p(p\overline{q})^{\dag} = p$.
    Combining Lemma \ref{corollary 2.5}(ii) and Lemma \ref{Lemma 2.3}(1), one can see that
    \begin{eqnarray*}
           (1-pq)(1-pq)^{\dag} & = & (1-pq)^{\dag}(1-pq)= (p-a)(p-a)^{\dag}+1-p \\
                               & = & p\overline{q}(p\overline{q}p)^{\dag}+1-p = p+1-p = 1.
    \end{eqnarray*}
    Hence $1-pq\in R^{-1}$ and $(\overline{q}p)^{\dag} = (1-pq)^{-1}p\overline{q}$.

(2) Since $(1-p)q+1-q = 1-pq\in R^{-1}$ by (1), it follows that $R = (1-pq)R \subseteq (1-p)R+(1-q)R$
    and hence $R = (1-p)R+(1-q)R$.
    Now $(\overline{q}p)^{\dag}R = pR$ follows by Theorem \ref{Theorem 3.4}(2).

(3) By Theorem \ref{Theorem 3.4}(3).
\end{proof}

\begin{corollary}\label{Corollary 3.6}
Let $p$ and $q$ be projections in a $*$-reducing ring $R$ such that $1-pq\in R^{\dag}$ and $pR+qR = R$, then

\emph{(1)} $(\overline{q}p)^{\dag}R = pR\cap (\overline{p}R+\overline{q}R)$;

\emph{(2)} $[1-(\overline{q}p)^{\dag}]R = qR$.
\end{corollary}

\begin{proof}
(1) See Theorem \ref{Theorem 3.4}(2).

(2) Since $1-pq\in R^{\dag}$, it follows that $(1-p)q\in R^{\dag}$ by Lemma \ref{corollary 2.5}(1)$\Leftrightarrow$(6).
    Moreover, $pR+qR = R$ implies $\overline{p}q\overline{p}R = \overline{p}R$ by Lemma \ref{Lemma 3.2}(4).
    Hence $\overline{p} = \overline{p}q\overline{p}r$ for some $r\in R$.
    Consequently, $\overline{p}q\overline{p}r = \overline{p} = \overline{p}^{*} = r^{*}\overline{p}q\overline{p}$.
    Now, for any $x\in \overline{p}R\cap \overline{q}R$, we have $x = \overline{p}x = \overline{q}x$
    and hence $x = \overline{p}x = \overline{p}q\overline{p}rx = r^{*}\overline{p}q\overline{p}x = 0$.
    Thus $\overline{p}R\cap \overline{q}R = {0}$.
    Therefore the result follow by Theorem \ref{Theorem 3.4}(3).
\end{proof}

\begin{remark}\label{Remark 3.8}
\emph{Let $p$ and $q$ be projections in a $*$-reducing ring $R$ such that $1-pq\in R^{\dag}$,
then $pR\cap qR={0}$ $\Leftrightarrow$ $1-pq\in R^{-1}$.
Indeed, $1-pq\in R^{-1}$ implies $R = (1-pq)R$.
In addition, $R = (1-pq)R \subseteq (1-p)R+(1-q)R$ since $1-pq = (1-p)q + 1 - q$.
Hence $R = (1-p)R+(1-q)R$.
Now, by the proof of Corollary \textsc{\ref{Corollary 3.6}}(2) one can see that $pR\cap qR = \{0\}$.
Conversely, when $pR\cap qR = \{0\}$ we have $1-pq\in R^{-1}$ by Corollary \textsc{\ref{Corollary 3.5}}(1).}
\end{remark}

\begin{theorem}\label{Theorem 3.9}
Let $p$ and $q$ be projections in a $*$-reducing ring $R$ such that $1-qp\in R^{\dag}$, then

\emph{(1)} $(1-qp)^{\dag}\overline{q}$ and $p(p+q-qp)^{\dag}$ are idempotents;

\emph{(2)} $[(1-qp)^{\dag}\overline{q}]R = [pR\cap (\overline{p}R+\overline{q}R)]\oplus^{\bot} (\overline{p}R\cap\overline{q}R)$;

\emph{(3)} $[1-(1-qp)^{\dag}\overline{q}]R = qR$;

\emph{(4)} $[p(p+q-qp)^{\dag}]R = pR$;

\emph{(5)} $[1-p(p+q-qp)^{\dag}]R = [(\overline{p}R+\overline{q}R)\cap qR]\oplus^{\bot} (\overline{p}R\cap\overline{q}R)$.
\end{theorem}

\begin{proof}
Since $1-qp\in R^{\dag}$, we have $p(1-q)$, $(1-p)q\in R^{\dag}$ by Lemma \ref{corollary 2.5}(3)$\Leftrightarrow$(5)$\Leftrightarrow$(6).

(1) By Lemma \ref{corollary 2.5}(iii) and Lemma \ref{Lemma 2.3}(1), it follows that
    \begin{eqnarray}
          (1-qp)^{\dag} = [(1-pq)^{\dag}]^{*} = (p-a)^{\dag}+b^{*}(p-a)^{\dag}+1-p.
    \end{eqnarray}
    Note that
   $(p-a)^{\dag}\overline{q} = (p-a)^{\dag}\overline{q}(p+1-p) = (p-a)^{\dag}(p-a)-(p-a)^{\dag}b$
   and
    \begin{eqnarray*}
          b^{*}(p-a)^{\dag}\overline{q} & = &  d^{\dag}b^{*}\overline{q} \quad\quad\quad\quad\quad\quad\quad\ (\mbox{see Lemma \ref{Lemma 2.3}(3)}\\
                                        & = & -d^{\dag}\overline{p}q(1-p)(1-q)\\
                                        & = & -d^{\dag}d +d^{\dag}dq  \\
                                        & = & -d^{\dag}d+(1-p)q. \quad\quad\ \ (\mbox{see Lemma \ref{Lemma 3.1}})
    \end{eqnarray*}
    Hence
    \begin{eqnarray}
           (1-qp)^{\dag}\overline{q} & = & (p-a)^{\dag}\overline{q}+b^{*}(p-a)^{\dag}\overline{q}+(1-p)\overline{q} \nonumber\\
                                     & = & (p-a)^{\dag}(p-a)-(p-a)^{\dag}b-dd^{\dag}+(1-p)q+(1-p)(1-q)\nonumber\\
                                     & = & (p-a)^{\dag}(p-a)-(p-a)^{\dag}b+1-p-dd^{\dag}.
    \end{eqnarray}
    Now, one can verify that
    $[(1-qp)^{\dag}\overline{q}]^{2} = (1-qp)^{\dag}\overline{q}$ by Lemma \ref{Lemma 2.3}(3).

    On the other hand, $(1-p)q\in R^{\dag}$ implies $p+q-qp\in R^{\dag}$ by Lemma \ref{corollary 2.5}(6)$\Leftrightarrow$(8).
    Replacing $p$ and $q$ by $1-p$ and $1-q$ respectively in (2.8),
    one can get
    \begin{eqnarray*}
          (p+q-qp)^{\dag}  & = & (\overline{p}q\overline{p})^{\dag}+p(1-q)(1-p)(\overline{p}q\overline{p})^{\dag}+p
                             = (\overline{p}q\overline{p})^{\dag} - pq\overline{p}(\overline{p}q\overline{p})^{\dag}+p,
    \end{eqnarray*}
    and hence
    \begin{eqnarray}
           p(p+q-qp)^{\dag} = p- pq\overline{p}(\overline{p}q\overline{p})^{\dag}.
    \end{eqnarray}
    Whence, it is easy to check that $p(p+q-qp)^{\dag}$ an idempotent.

(2) By (2.4) and (2.6) we can see that
    $pR\cap (\overline{p}R+\overline{q}R) = y^{\prime}R = p\overline{q}(p\overline{q}p)^{\dag}R$ and
    $\overline{p}R\cap\overline{q}R = p^{\prime}R = [\overline{p}-\overline{p}(\overline{p}q)^{\dag}]R$,
    where
    $y^{\prime} = p\overline{q}(p\overline{q}p)^{\dag}$,
    $p^{\prime} = \overline{p}-\overline{p}(\overline{p}q)^{\dag}$.
    Hence $(p^{\prime})^*y^{\prime} = 0$.
    Thereby
    \begin{eqnarray}
          [pR\cap (\overline{p}R+\overline{q}R)]\bot\, [\overline{p}R\cap\overline{q}R].
    \end{eqnarray}
    Since
    $\overline{y^{\prime}}p^{\prime} = [1-p\overline{q}(p\overline{q}p)^{\dag}][\overline{p}-\overline{p}(\overline{p}q)^{\dag}]
                                     = \overline{p}-\overline{p}(\overline{p}q)^{\dag}
                                     = p^{\prime}$,
    we have $(\overline{y^{\prime}}p^{\prime})^{\dag} = p^{\prime}$.
    Let $x_{1} = y^{\prime}+\overline{y^{\prime}}(\overline{y^{\prime}}p^{\prime})^{\dag}$.
    Replacing $p$ and $q$ by $y^{\prime}$ and $p^{\prime}$ respectively in Lemma \ref{Lemma 3.2}(2),
    we have
    \begin{eqnarray}
           x_{1}R = y^{\prime}R+p^{\prime}R
                  = [pR\cap (\overline{p}R+\overline{q}R)]+ [\overline{p}R\cap\overline{q}R].
    \end{eqnarray}
    Note that
    \begin{eqnarray}
          x_{1} = y^{\prime}+\overline{y^{\prime}}(\overline{y^{\prime}}p^{\prime})^{\dag}
                = y^{\prime}+p^{\prime} = p\overline{q}(p\overline{q}p)^{\dag}+\overline{p}-\overline{p}(\overline{p}q)^{\dag}
                = p\overline{q}(p\overline{q}p)^{\dag}+\overline{p}-\overline{p}q(\overline{p}q\overline{p})^{\dag}.
    \end{eqnarray}
    By Lemma \ref{Lemma 2.2}(1) and Lemma \ref{Lemma 2.3}(2)(3), one can see that
    \begin{eqnarray*}
           x_{1} & \overset{(2.13)}{=\!\!\!=\!\!\!=\!\!\!=} & p\overline{q}(p\overline{q}p)^{\dag}+\overline{p}-\overline{p}q(\overline{p}q\overline{p})^{\dag}\\
                 & =                                        & [(p-a)^{\dag}(p-a)-(p-a)^{\dag}b+1-p-dd^{\dag}][(p-a)-b^{*}+1-p-dd^{\dag}]\\
                 & \overset{(2.9)}{=\!\!\!=\!\!\!=\!\!\!=} & (1-qp)^{\dag}\overline{q}[(p-a)-b^{*}+1-p-dd^{\dag}]
    \end{eqnarray*}
    and
    \begin{eqnarray*}
          (1-qp)^{\dag}\overline{q} & \overset{(2.9)}{=\!\!\!=\!\!\!=\!\!\!=} &
                                       (p-a)^{\dag}(p-a)-(p-a)^{\dag}b+1-p-dd^{\dag} \\
                                    & =                                        &
                                       [(p-a)(p-a)^{\dag}+\overline{p}-dd^{\dag}][(p-a)(p-a)^{\dag}-(p-a)^{\dag}b+\overline{p}-dd^{\dag}]\\
                                    & \overset{(2.13)}{=\!\!\!=\!\!\!=\!\!\!=} &
                                       x_{1}[(p-a)(p-a)^{\dag}-(p-a)^{\dag}b+\overline{p}-dd^{\dag}].
    \end{eqnarray*}
    Whence $x_{1}R = (1-qp)^{\dag}\overline{q}R$.
    In addition, (2.11) and (2.12) imply
    $$x_{1}R = [pR\cap (\overline{p}R+\overline{q}R)]+[\overline{p}R\cap\overline{q}R]
            = [pR\cap (\overline{p}R+\overline{q}R)]\oplus^{\bot} [\overline{p}R\cap\overline{q}R].$$
    Thus
    $$(1-qp)^{\dag}\overline{q}R = [pR\cap (\overline{p}R+\overline{q}R)]\oplus^{\bot} [\overline{p}R\cap\overline{q}R].$$

(3) Since $(1-p)q\in R^{\dag}$ we have $(1-p)q(1-p)\in R^{\dag}$ by Lemma \ref{corollary 2.5}(6)$\Leftrightarrow$(7).
    Let $z = 1-qp-dd^{\dag}$.
    We claim that $[(1-qp)^{\dag}\overline{q}]^{\dag} = z$.

    Indeed, note that
    \begin{eqnarray*}
               &   &  [(1-qp)^{\dag}\overline{q}]z                                                   \\
               & \overset{(2.9)}{=\!\!\!=\!\!\!=\!\!\!=} &
                   [(p-a)^{\dag}(p-a)-(p-a)^{\dag}b+1-p-dd^{\dag}][1-qp-dd^{\dag}]                 \\
               & = & (p-a)^{\dag}(p-a)(1-qp)-(p-a)^{\dag}b(1-qp)+(1-p)(1-qp)                  \\
               &   & -dd^{\dag}(1-qp)+(p-a)^{\dag}b  \quad\quad\quad\quad\quad\  (\mbox{see Lemma \ref{Lemma 2.3}(2)})    \\
               & = & (p-a)+(p-a)^{\dag}bqp+1-p-(1-p)qp-dd^{\dag}+dd^{\dag}qp                         \\
               & = & (p-a)+(p-a)^{\dag}bb^*+1-p-dd^{\dag}  \quad (\mbox{see Lemma \ref{Lemma 2.3}(2)})   \\
               & = & (p-a)^{\dag}(p-a)+1-p-dd^{\dag}, \quad \quad\quad\ \ \! (\mbox{see Lemma \ref{Lemma 2.2}(1)})
    \end{eqnarray*}
    where $[(p-a)^{\dag}(p-a)]^* = (p-a)^{\dag}(p-a)$, $(1-p)^* = 1-p$ and $(dd^{\dag})^* = dd^{\dag}$.
    Hence $[(1-qp)^{\dag}\overline{q}z]^* = (1-qp)^{\dag}\overline{q}z$ and
    \begin{eqnarray*}
                 &   &  [(1-qp)^{\dag}\overline{q}]z [(1-qp)^{\dag}\overline{q}]                      \\
                 & \overset{(2.9)}{=\!\!\!=\!\!\!=\!\!\!=} &
                   [(p-a)^{\dag}(p-a)+1-p-dd^{\dag}][(p-a)^{\dag}(p-a)-(p-a)^{\dag}b+1-p-dd^{\dag}]  \\
                 & = & (p-a)^{\dag}(p-a)-(p-a)^{\dag}b+1-p-dd^{\dag}                                   \\
                 & \overset{(2.9)}{=\!\!\!=\!\!\!=\!\!\!=} & (1-qp)^{\dag}\overline{q}.
    \end{eqnarray*}
    Similarly, it follows from
    \begin{eqnarray}
                &   &  z[(1-qp)^{\dag}\overline{q}]                                                   \nonumber\\
                & \overset{(2.9)}{=\!\!\!=\!\!\!=\!\!\!=} &
                   [1-qp-dd^{\dag}][(p-a)^{\dag}(p-a)-(p-a)^{\dag}b+1-p-dd^{\dag}]                    \nonumber\\
                & = & (1-qp)p(p-a)^{\dag}(p-a)-(1-qp)p(p-a)^{\dag}b+1-p-dd^{\dag}                       \nonumber\\
                & = & (1-q)p-(1-q)ppq(1-p)d^{\dag}+1-p-dd^{\dag} \quad\quad (\mbox{see Lemma \ref{Lemma 3.1} and \ref{Lemma 2.3}(3)})\nonumber\\
                & = & (1-q)p+(1-q)(1-p)q(1-p)d^{\dag}+1-p-dd^{\dag}                                    \nonumber\\
                & = & (1-q)p-q(1-p)q(1-p)d^{\dag}+1-p                                   \nonumber\\
                & = & (1-q)p-q(1-p)+1-p    \quad\quad \quad\quad \quad\quad\quad\quad\quad\ \  (\mbox{see Lemma \ref{Lemma 3.1}})    \nonumber\\
                & = & 1-q
    \end{eqnarray}
    that $[z(1-qp)^{\dag}\overline{q}]^* = z(1-qp)^{\dag}\overline{q}$ and
    \begin{eqnarray*}
                 z[(1-qp)^{\dag}\overline{q}]z & = & (1-q)[1-qp-dd^{\dag}]            \\
                                         & = & 1-q-dd^{\dag}+qdd^{\dag}                     \\
                                         & = & 1-q-dd^{\dag}+q(1-p)\quad\quad (\mbox{see Lemma \ref{Lemma 3.1}})      \\
                                         & = & 1-qp-dd^{\dag} = z.
    \end{eqnarray*}

    Now, by (2.14) we have
    \begin{eqnarray*}
                 1-[(1-qp)^{\dag}\overline{q}]^{\dag}[(1-qp)^{\dag}\overline{q}] & = & 1-z[(1-qp)^{\dag}\overline{q}] = 1-(1-q) = q.
    \end{eqnarray*}
    Whence $[1-(1-qp)^{\dag}\overline{q}]R = \{1-[(1-qp)^{\dag}\overline{q}]^{\dag}[(1-qp)^{\dag}\overline{q}]\}R = qR$
    by Lemma \ref{Lemma 3.3}.

(4) First, $p(1-q)\in R^{\dag}$ implies $p-pqp\in R^{\dag}$ by Lemma \ref{corollary 2.5}(2)$\Leftrightarrow$(3).
    Let
    \begin{eqnarray*}
           z^{\prime} = 2p-qp-(p-a)(p-a)^{\dag} = p+p-a-b^{*}-(p-a)(p-a)^{\dag}.
    \end{eqnarray*}
    We claim $z^{\prime}$ is the MP inverse of $p(p+q-qp)^{\dag}$.

    Indeed, it follows that
    \begin{eqnarray}
              &   & p(p+q-qp)^{\dag}z^{\prime}                                                                                        \nonumber\\
              & \overset{(2.10)}{=\!\!\!=\!\!\!=\!\!\!=} &
                 [p- pq\overline{p}(\overline{p}q\overline{p})^{\dag}][p+p-a-b^{*}-(p-a)(p-a)^{\dag}]                                  \nonumber\\
              & = & p+p-a-(p-a)(p-a)^{\dag}+pq\overline{p}(\overline{p}q\overline{p})^{\dag}b^{*}                                       \nonumber\\
              & = & p+p-a-(p-a)(p-a)^{\dag}+(p-a)^{\dag}bb^{*} \quad\quad\quad\quad\quad\quad\quad\quad\  (\mbox{see Lemma \ref{Lemma 2.3}(3)})\nonumber\\
              & = & p+p-a-(p-a)(p-a)^{\dag}+(p-a)^{\dag}[p-a-(p-a)^2] \quad\quad (\mbox{see Lemma \ref{Lemma 2.2}(1)})                        \nonumber\\
              & = & p.
    \end{eqnarray}
    This guarantees $[p(p+q-qp)^{\dag}z^{\prime}]^* = p(p+q-qp)^{\dag}z^{\prime}$ and
    \begin{eqnarray*}
                 p(p+q-qp)^{\dag}z^{\prime}p(p+q-qp)^{\dag} = pp(p+q-qp)^{\dag} = p(p+q-qp)^{\dag}.
    \end{eqnarray*}
    Similarly, note that
    \begin{eqnarray}
                 &   & z^{\prime}p(p+q-qp)^{\dag}                                                                        \nonumber       \\
                 & \overset{(2.10)}{=\!\!\!=\!\!\!=\!\!\!=} &
                 [p+p-a-b^{*}-(p-a)(p-a)^{\dag}][p- pq\overline{p}(\overline{p}q\overline{p})^{\dag}]                        \nonumber         \\
                 & = & p+p-a-b^{*}-(p-a)(p-a)^{\dag}-(p-a)bd^{\dag}+b^{*}b(\overline{p}q\overline{p})^{\dag}\quad
                             (\mbox{see Lemma \ref{Lemma 2.3}(1)})   \nonumber \\
                 & = & p+p-a-b^{*}-(p-a)(p-a)^{\dag}-b +b^{*}bd^{\dag}\quad\quad\quad\quad\quad\quad\ \ (\mbox{see Lemma \ref{Lemma 2.3}(1)(3)}) \nonumber \\
                 & = & 2p-a-b^{*}-(p-a)(p-a)^{\dag}-b+dd^{\dag}-d, \  \quad\quad\quad\quad\quad\quad
                             (\mbox{see Lemma \ref{Lemma 2.2}(2)}) \nonumber \\
                 & = & 2p-q-(p-a)(p-a)^{\dag}+dd^{\dag}.
    \end{eqnarray}
    So we have $[z^{\prime}p(p+q-qp)^{\dag}]^* = z^{\prime}p(p+q-qp)^{\dag}$
    and
    \begin{eqnarray*}
              &   & z^{\prime}p(p+q-qp)^{\dag}z^{\prime}                              \\
              & = & [2p-q-(p-a)(p-a)^{\dag}+dd^{\dag}][p+p-a-b^{*}-(p-a)(p-a)^{\dag}] \\
              & = & 2p+(p-a)-2(p-a)(p-a)^{\dag}-q(p+p-a-b^*)+q(p-a)(p-a)^{\dag}-dd^{\dag}b^{*}\\
              & = & 2p+(p-a)-(p-a)(p-a)^{\dag}-(1-q)(p-a)(p-a)^{\dag}-qp-b^{*} \quad\quad (\mbox{see Lemma \ref{Lemma 2.3}(2)})\\
              & = & 2p+(p-a)-(p-a)(p-a)^{\dag}-(1-q)p-qp-b^{*} \quad\quad\quad\quad\quad\quad\quad\quad (\mbox{see Lemma \ref{Lemma 3.1}})\\
              & = & p+p-a-(p-a)(p-a)^{\dag}-b^{*} = z^{\prime}.
    \end{eqnarray*}

    Now, it follows from (2.15) that
    $[p(p+q-qp)^{\dag}][p(p+q-qp)^{\dag}]^{\dag} = [p(p+q-qp)^{\dag}]z^{\prime} = p$.
    Consequently, $$[p(p+q-qp)^{\dag}]R =[p(p+q-qp)^{\dag}][p(p+q-qp)^{\dag}]^{\dag}R =pR.$$

(5) By (2.2) we have $q^{\prime} = \overline{p}+p(p\overline{q})^{\dag}$.
    In view of Lemma \ref{Lemma 3.1} and Lemma \ref{Lemma 3.2}(1),
    one can see that
    $$q^{\prime}\overline{q} = [\overline{p}+p(p\overline{q})^{\dag}]\overline{q}
                             = (1-p)(1-q)+p\overline{q}(p\overline{q}p)^{\dag}\overline{q}
                             = (1-p)(1-q)+p\overline{q}
                             = \overline{q}.$$
    Hence $(q^{\prime}\overline{q})^{\dag} = \overline{q}$.
    Replacing $p$ by $q^{\prime}$ in Lemma \ref{Lemma 3.2}(3),
    one can see that $q^{\prime}-q^{\prime}(q^{\prime}\overline{q})^{\dag}$ is a projection and
    \begin{eqnarray}
             [q^{\prime}-q^{\prime}(q^{\prime}\overline{q})^{\dag}]R
           = q^{\prime}R\cap qR \overset{(2.3)}{=\!\!\!=\!\!\!=\!\!\!=} (\overline{p}R+\overline{q}R)\cap qR.
    \end{eqnarray}
    Let $y_{1} = q^{\prime}-q^{\prime}(q^{\prime}\overline{q})^{\dag}$, then we have
    \begin{eqnarray}
           y_{1} = \overline{p}+p(p\overline{q})^{\dag}-q^{\prime}\overline{q}
                 = \overline{p}+p(p\overline{q})^{\dag}-\overline{q} = q-p+p(p\overline{q})^{\dag}
    \end{eqnarray}
    and
    \begin{eqnarray*}
    \overline{y_{1}}p^{\prime} & = & [p-p(p\overline{q})^{\dag}+\overline{q}][\overline{p}-\overline{p}(\overline{p}q)^{\dag}]
                                        \quad\quad\ \ \quad\quad(\mbox{see (2.5),(2.18)})\\
                               & = & [p-p\overline{q}(p\overline{q}p)^{\dag}+\overline{q}][\overline{p}-\overline{p}q(\overline{p}q\overline{p})^{\dag}]
                                        \quad\quad\ \ (\mbox{see Lemma \ref{Lemma 3.2}(1)})\\
                               & = & (1-q)\overline{p}-(1-q)\overline{p}q(\overline{p}q\overline{p})^{\dag}\\
                               & = & (1-q)\overline{p}- (1-q)\overline{p}q\overline{p}(\overline{p}q\overline{p})^{\dag}\\
                               & = & (1-q)\overline{p}-(\overline{p}q\overline{p})(\overline{p}q\overline{p})^{\dag}
                                        +q(\overline{p}q\overline{p})(\overline{p}q\overline{p})^{\dag}\\
                               & = & (1-q)\overline{p}-(\overline{p}q\overline{p})(\overline{p}q\overline{p})^{\dag}+q\overline{p}
                                        \quad\quad\quad\quad\ (\mbox{see Lemma \ref{Lemma 3.1}})\\
                               & = & \overline{p}-(\overline{p}q\overline{p})(\overline{p}q\overline{p})^{\dag} \\
                               & = & p^{\prime}, \quad\quad\quad\quad\quad\quad\quad\quad\quad\quad\quad\quad\quad\quad\ \ \
                                                 (\mbox{see (2.5) and Lemma \ref{Lemma 3.2}(1)})
    \end{eqnarray*}
    hence $(\overline{y_{1}}p^{\prime})^{\dag} = p^{\prime}$.
    Replacing $p$ and $q$ by $y_{1}$ and $p^{\prime}$ respectively in Lemma \ref{Lemma 3.2}(2),
    one can see that $[y_{1}+\overline{y_{1}}(\overline{y_{1}}p^{\prime})^{\dag}]R =  y_{1}R+p^{\prime}R$.

    Next, let $x_{2} =  y_{1}+\overline{y_{1}}(\overline{y_{1}}p^{\prime})^{\dag}$.
    By (2.5), (2.18) and Lemma \ref{Lemma 3.2}(1),
    \begin{eqnarray}
           x_{2} & = &  y_{1}+\overline{y_{1}}(\overline{y_{1}}p^{\prime})^{\dag}
                   =    y_{1}+p^{\prime} \nonumber\\
                 & = &  q-p+p(p\overline{q})^{\dag}+\overline{p}-\overline{p}(\overline{p}q)^{\dag}\nonumber\\
                 & = &  q-2p+1+(p\overline{q}p)(p\overline{q}p)^{\dag}-(\overline{p}q\overline{p})(\overline{p}q\overline{p})^{\dag}.
    \end{eqnarray}
    Moreover, we have
    \begin{eqnarray*}
           (p^{\prime})^{*}y_{1} & = & [(q\overline{p})^{\dag}\overline{p}-\overline{p}][q-p+p(p\overline{q})^{\dag}]
                                        \ \ \  \ \ (\mbox{see (2.5) and (2.18)}) \\
                                 & = & (\overline{p}q\overline{p})^{\dag}\overline{p}q\overline{p}q-\overline{p}q
                                        \quad\quad\ \ \quad\quad\quad\ \ (\mbox{see Lemma \ref{Lemma 3.2}(1)})\\
                                 & = & 0,   \quad\quad\quad\quad\ \quad\quad\quad\quad\quad\quad\quad\  \ \ (\mbox{see Lemma \ref{Lemma 3.1}})
    \end{eqnarray*}
    from which it follows that $y_{1}R \bot p^{\prime}R$.
    Whence
    \begin{eqnarray}
             x_{2}R = y_{1}R\oplus^{\bot}p^{\prime}R.
    \end{eqnarray}
    On the other hand,
    \begin{eqnarray}
          &   &  1-[p(p+q-qp)^{\dag}]^{\dag}[p(p+q-qp)^{\dag}] \nonumber\\
          & = &  1-z^{\prime}[p(p+q-qp)^{\dag}]                \nonumber\\
          & \overset{(2.16)}{=\!\!\!=\!\!\!=\!\!\!=} &
                 1-[2p-q-(p-a)(p-a)^{\dag}+dd^{\dag}]             \nonumber\\
          & = &  (p\overline{q}p)(p\overline{q}p)^{\dag}+q-2p+1-(\overline{p}q\overline{p})(\overline{p}q\overline{p})^{\dag}.
    \end{eqnarray}
    According to (2.19) and (2.21), we obtain
    \begin{eqnarray}
           \{1-[p(p+q-qp)^{\dag}]^{\dag}[p(p+q-qp)^{\dag}]\}R = x_{2}R.
    \end{eqnarray}
    Finally, (2.22), (2.20), (2.17), (2.6) and Lemma \ref{Lemma 3.3} imply
    \begin{eqnarray*}
           [1-p(p+q-qp)^{\dag}]R & = & \{1-[p(p+q-qp)^{\dag}]^{\dag}[p(p+q-qp)^{\dag}]\}R          \\
                                 & = & x_{2}R = y_{1}R\oplus^{\bot}p^{\prime}R                     \\
                                 & = & [(\overline{p}R+\overline{q}R)\cap qR]\oplus^{\bot} (\overline{p}R\cap\overline{q}R).
    \end{eqnarray*}
    This completes the proof.
\end{proof}

The following corollary is a special case of Theorem \ref{Theorem 3.9}.

\begin{corollary}\label{Corollary 3.10}
Let $p$ and $q$ be projections in a $*$-reducing ring $R$ such that $pR\oplus qR = R$, then

\emph{(1)} $1-qp, p+q-qp\in R^{-1}$;

\emph{(2)} $[(1-qp)^{-1}\overline{q}]R = [p(p+q-qp)^{-1}]R = pR$;

\emph{(3)} $[1-(1-qp)^{-1}\overline{q}]R = [1-p(p+q-qp)^{-1}]R = qR$.
\end{corollary}

\begin{proof}
(1) By \cite[Theorem 4.4]{Koliha Rakocevic 2003}.

(2) Since $pR\oplus qR = R$, we have $\overline{p}R\oplus \overline{q}R = R$ by \cite[Lemma 2.1]{Koliha Rakocevic 2003}.
    Now $[(1-qp)^{-1}\overline{q}]R$ = $[p(p+q-qp)^{-1}]R = pR$ follows by Theorem \ref{Theorem 3.9}.

(3) Similar to (2).
\end{proof}

\begin{theorem}\label{Theorem 3.11}
Let $p$ and $q$ be projections in a $*$-reducing ring $R$ such that $1-qp\in R^{\dag}$.

\emph{(1)} $(\overline{q}p)^{\dag} = (1-qp)^{\dag}\overline{q}$ if and only if $pR + qR = R$.

\emph{(2)} $(\overline{q}p)^{\dag} = p(p+q-qp)^{\dag}$ if and only if $pR \cap qR = \{0\}$.

\emph{(3)} $(1-qp)^{\dag}\overline{q} = p(p+q-qp)^{\dag}$ if and only if $pR \oplus qR = R$.
\end{theorem}

\begin{proof}
(1) First, we prove that $(\overline{q}p)^{\dag} = (1-qp)^{\dag}\overline{q}$ if and only if $1-p = dd^{\dag}$.
    Indeed, $1-qp\in R^{\dag}$ and (2.9) imply
    $$(1-qp)^{\dag}\overline{q} = (p-a)(p-a)^{\dag}- (p-a)^{\dag}b+1-p-dd^{\dag}.$$
    Meanwhile, by $1-qp\in R^{\dag}$ and Lemma \ref{corollary 2.5}(4)$\Leftrightarrow$(5) we have $\overline{q}p\in R^{\dag}$.
    Hence, according to Lemma \ref{Lemma 3.2}(1), one can see that
    \begin{eqnarray}
           (\overline{q}p)^{\dag} & = & [(p\overline{q})^{\dag}]^{*}
                                    =   (p\overline{q}p)^{\dag}\overline{q}
                                    =   (p\overline{q}p)^{\dag}\overline{q}(p+1-p)\nonumber\\
                                  & = & (p-a)^{\dag}(p-a)+(p-a)^{\dag}\overline{q}(1-p)\nonumber\\
                                  & = & (p-a)^{\dag}(p-a)- (p-a)^{\dag}b.
    \end{eqnarray}
    Thus, $(\overline{q}p)^{\dag} = (1-qp)^{\dag}\overline{q}$ if and only if $1-p = dd^{\dag}$.

    Next, we prove that $1-p = dd^{\dag}$ if and only if $pR + qR = R$.
    Indeed, by \cite[Lemma 4.1(4)]{Benitez CvetkovicIlic},
    $dd^{\dag}R = dR = (1-p)R$ if and only $1-p = dd^{\dag}$.
    In view of $1-qp\in R^{\dag}$ and Lemma \ref{corollary 2.5}(5)$\Leftrightarrow$(6), we get $(1-p)q\in R^{\dag}$.
    According to Lemma \ref{Lemma 3.2}(4), we know that $pR + qR = R$ if and only if $dd^{\dag}R = dR = (1-p)R$ if and only if $1-p = dd^{\dag}$.

(2) Since $1-qp\in R^{\dag}$, we have $p(p+q-qp)^{\dag} = p-pq\overline{p}(\overline{p}q\overline{p})^{\dag}$ by (2.10).
    By Lemma \ref{corollary 2.5}(3)$\Leftrightarrow$(5)$\Leftrightarrow$(6), it follows that $p(1-q), (1-p)q\in R^{\dag}$.
    Whence
    $$p(p+q-qp)^{\dag} = p-pq\overline{p}(\overline{p}q\overline{p})^{\dag} = p-(p-a)^{\dag}b$$
    by Lemma \ref{Lemma 2.3}(3).
    Combining this with (2.23), we can see that $(\overline{q}p)^{\dag} = p(p+q-qp)^{\dag}$ if and only if $p=(p-a)(p-a)^{\dag}$.

    On the other hand, we have
    \begin{eqnarray*}
           pR \cap qR =  \{0\} \Leftrightarrow p=(p-a)(p-a)^{\dag}
    \end{eqnarray*}
    by Lemma \ref{Lemma 3.2}(1)(3).
    Therefore, $(\overline{q}p)^{\dag} = p(p+q-qp)^{\dag}$ if and only if $pR \cap qR = \{0\}$.

(3) is an immediate consequence of (1) and (2).
\end{proof}

We note that the hypothesis that $R$ is $*$-reducing can not be removed from Theorem \ref{Theorem 3.4} and Theorem \ref{Theorem 3.9}
(see \cite[Example 6]{Zhang Zhang Chen Wang}).
But we don't know whether this hypothesis can be removed from Theorem \ref{Theorem 3.11}.

We complete this section with the following result.

\begin{theorem}\label{Theorem 3.13}
Let $p$ and $q$ be two projections in $R$ such that $p(1-q), (1-p)q \in R^{\dag}$, then
$[1-p\overline{q}(p\overline{q}p)^{\dag}-\overline{p}q(\overline{p}q\overline{p})^{\dag}]R
  = (pR\cap qR)\oplus^{\bot} (\overline{p}R\cap\overline{q}R)$
where $1-p\overline{q}(p\overline{q}p)^{\dag}-\overline{p}q(\overline{p}q\overline{p})^{\dag}$ is a projection.
\end{theorem}

\begin{proof}
Since $p(1-q)\in R^{\dag}$,
it follows that $pR\cap qR$ is generated by the projection $p-p(p\overline{q})^{\dag}$ by Lemma \ref{Lemma 3.2}(3).
Similarly, $\overline{p}R\cap \overline{q}R$ is generated by the projection $\overline{p}-\overline{p}(\overline{p}q)^{\dag}$
since $(1-p)q\in R^{\dag}$.
Let
$p_{1} = p-p(p\overline{q})^{\dag}$ and $q_{1} = \overline{p}-\overline{p}(\overline{p}q)^{\dag}$ for short.
Then we have
$\overline{p_{1}}q_{1} = [1-p+p(p\overline{q})^{\dag}][\overline{p}-\overline{p}(\overline{p}q)^{\dag}]
                       = \overline{p}-\overline{p}(\overline{p}q)^{\dag} = q_{1}$
and hence
$(\overline{p_{1}}q_{1})^{\dag} = q_{1}$.
Replacing $p$ and $q$ by $p_{1}$ and $q_{1}$ respectively in Lemma \ref{Lemma 3.2}(2),
we can see that
$(pR\cap qR)+(\overline{p}R\cap \overline{q}R)$ is generated by the projection $p_{1}+\overline{p_{1}}(\overline{p}_{1}q_{1})^{\dag}$.

Note that $p_{1}R\bot q_{1}R$ since $q_{1}p_{1} = 0$.
Hence $p_{1}R+q_{1}R = p_{1}R\oplus^{\bot}q_{1}R$, i.e.,
$(pR\cap qR)+(\overline{p}R\cap \overline{q}R) = (pR\cap qR)\oplus^{\bot}(\overline{p}R\cap \overline{q}R)$.

Finally, we have
\begin{eqnarray*}
             p_{1}+\overline{p_{1}}(\overline{p}_{1}q_{1})^{\dag}
       & = & p_{1}+q_{1}
         =   p-p(p\overline{q})^{\dag}+1-p-\overline{p}(\overline{p}q)^{\dag} \\
       & = & 1-p\overline{q}(p\overline{q}p)^{\dag}-\overline{p}q(\overline{p}q\overline{p})^{\dag}
\end{eqnarray*}
by Lemma \ref{Lemma 3.2}(1).
\end{proof}

\vskip 2mm

\centerline {\bf ACKNOWLEDGMENTS} This research is supported by the National Natural Science Foundation of China (11201063),
the Specialized Research Fund for the Doctoral Program of Higher Education (20120092110020),
the Natural Science Foundation of Jiangsu Province(BK2010393) and the Foundation of Graduate Innovation Program of Jiangsu Province(CXZZ12-0082).

\end{document}